\documentclass{amsproc}

\usepackage[all,cmtip]{xy}
\usepackage{tikz}
\usepackage{amssymb}
\usepackage{amsthm}
\usepackage{amsmath}
\usepackage{hyperref}
\usepackage{enumerate}
\usepackage{multirow}
\usepackage{lscape}
\usepackage{float}
 \usepackage{graphicx,rotating}

\usepackage{footnote}

\newtheorem{theorem}{Theorem}[section]
\newtheorem{lemma}[theorem]{Lemma}

\theoremstyle{definition}
\newtheorem{definition}[theorem]{Definition}

\newtheorem{proposition}[theorem]{Proposition}
\newtheorem{corollary}[theorem]{Corollary}

\theoremstyle{remark}
\newtheorem{remark}[theorem]{Remark}

\numberwithin{equation}{section}

\begin{document}

\title{Representation stability for the pure cactus group}

\author[R. Jim\'enez Rolland]{Rita Jim\'enez Rolland} 
\author[J. Maya Duque]{Joaqu\' in Maya Duque*\footnote{* The second author is grateful for the support of the CONACYT grant 168093-F.}}

\address{ Centro de Ciencias Matem\'aticas, Universidad Nacional Aut\'onoma de M\'exico,  Morelia, Michoac\'an, M\'exico C.P. 58089.}
\curraddr{}
\email{rita@matmor.unam.mx}

\address{CINVESTAV, Colonia Pedro Zacatenco, M\'exico, D.F. C.P. 07360.}
\curraddr{}
\email{jmaya@math.cinvestav.mx}


\subjclass[2000]{Primary }

\date{\today}

\maketitle


\begin{abstract} 
The fundamental group of the real locus of the Deligne-Mumford compactification of the moduli space of rational curves with $n$ marked points, the pure cactus group, resembles the pure braid group in many ways. 
As it is the case for several ``pure braid like'' groups, it is known that its cohomology ring is generated by its first cohomology. In this note we survey what the $FI$-module theory developed by Church, Ellenberg and Farb can tell us about those examples.  As a consequence we obtain uniform representation stability for the sequence of cohomology groups of the pure cactus group.
\end{abstract}

\section{Introduction}
In this  paper we survey the principal notions and consequences of the $FI$-module theory introduced by Church, Ellenberg and Farb in \cite{CEF}. Our main objective is to show how the theory can be readily applied to certain sequences of groups and spaces with cohomology rings that have the structure of a graded $FI$-algebra and are known to be generated by the first cohomology group.  We revisit the collection of examples in Section \ref{OTROS} that have in common  this behavior (see Table \ref{EXA}).   


To illustrate our discussion we work out the details for the sequence of spaces $\{M_n\}$. The space $M_n:=\overline{\mathcal{M}}_{0,n}(\mathbb{R})$ is the real locus of the Deligne-Mumford compactification $\overline{\mathcal{M}}_{0,n}$ of the moduli space of rational curves with $n$ marked points. We recall the precise definition of $M_n$ in Section \ref{REALMODULI} below. 


As pointed out by Etingof--Henriques--Kamnitzer--Rains in \cite{EHKR}, the space $M_n$ resembles in many ways the configuration space of $n-1$ distinct points in the plane $\mathcal{F}(\mathbb{C},n-1)$. For instance, the spaces $M_n$ are smooth manifolds and Eilenberg-Mac Lane spaces (\cite{DJS}). In particular, it follows that the cohomology of the {\it pure cactus groups} $\pi_1(M_n)$ coincides with the cohomology of $M_n$. Another similarity, there exist maps from $M_n$ to $M_{n-1}$ which can be described as ``forgetting one marked point" and from  $M_n$ to $M_{n+1}$ which can be thought of ``inserting a bubble at a marked point". These maps endow the collection of spaces $\{ M_n \}$ with the structure of a simplicial space. Furthermore, as we recall in  ~\autoref{thm:cohomology ring}, the rational cohomology of $M_n$ is generated multiplicatively by the $1$-dimensional classes coming from the simplest of such spaces, namely, $M_4=S^1$. 

On the other hand, there are also notable differences between the spaces $M_n$ and $\mathcal{F}(\mathbb{C},n-1)$. For instance, for $n>5$ the space $M_n$ is not formal \cite{EHKR}. In contrast to the configuration spaces, the forgetful maps $M_n\rightarrow M_{n-1}$ are not fibrations. It is not known whether the pure cactus group $\pi_1 (M_n)$ is residually nilpotent.

In the present note we show that the spaces $M_n$ share one more property with the configuration spaces and the examples considered in Section \ref{OTROS}. Namely, we see in Section \ref{FIREP} that  their rational cohomology ring has the structure of a graded $FI$-algebra which is a finitely generated $FI$-module in each degree. As main consequence of this approach we obtain, as in the case of configuration spaces, uniform representation stability in the sense of \cite{CF}.

\begin{theorem}\label{thm:main}
The sequence $\{H^i(M_{n},\mathbb{Q})\}_n$ of $S_{n}$-representations is uniformly representation stable for all $i\geq0$ and the stability holds for $n\geq 6i.$
\end{theorem}
 
Moreover,  we show that the characters of the $S_n$-representations are strongly constrained.

\begin{theorem}\label{thm:CHARCACTUS}
For $i\geq 0$ and $n\geq 6i$ the character of the $S_n$-representation $H^i(M_n,\mathbb{Q})$ is given by a unique character polynomial $P_i$ of degree $\leq 3i$. 
\end{theorem}

A {\it character polynomial} is a polynomial in the cycle-counting functions $X_l$, where  $X_l(\sigma)$ counts the  number of $l$-cycles in the cycle decomposition of $\sigma\in\bigsqcup_n S_n$.   For instance, for $n\geq 4$ it is known that  $H^1(M_n,\mathbb{Q})=\bigwedge^3 \mathcal{H}_n$, where  $\mathcal{H}_n$ is the standard representation of $S_n$ of dimension $n-1$. Its character is given by the character polynomial of degree $3$: 
$$\chi_{H^1(M_n,\mathbb{Q})}=\binom{X_1}{3}+X_3-X_2X_1-\binom{X_1}{2}+X_2+X_1-1.$$
Notice that for $n\geq 4$, the value of the character in any permutation $\sigma\in S_n$ only depends on cycles of length $1$, $2$ and $3$ in the cycle decomposition of $\sigma$.  More generally, Theorem \ref{thm:CHARCACTUS}  implies that  the character of the $S_n$-representation $H^i(M_n,\mathbb{Q})$ thought of as a function of  $\sigma\in S_n$  is completely independent to the number of $r$-cycles for all $r > 3i$.

Furthermore, we recover the following result first proven in \cite[Theorem 6.4]{EHKR}.

\begin{corollary}\label{BETTICACTUS}
For every $i\geq 1$, there exists a unique polynomial $p_i\in\mathbb{Q}[t]$ of degree $\leq 3i$ such that for $n\geq 6i$, the $i$th Betti number of $M_n$  is given by $b_i(M_n)=p_i(n)$.
\end{corollary} 

In particular,  for $n\geq 4$ the first Betti number $b_1(M_n)$  is the cubic polynomial in $n$:
  
$$\dim_\mathbb{Q}\big(H^1(M_n,\mathbb{Q})\big)=\chi_{H^1(M_n,\mathbb{Q})}(\text{id})= \binom{n}{3}-\binom{n}{2}+n-1=\frac{(n-1)(n-2)(n-3)}{6}.$$

More is actually known about these representations. In \cite[Theorem 1.1]{RAINS} Rains obtained an explicit formula for the graded character of the $S_n$-action on $H^*(M_n,\mathbb{Q})$. In particular, his result  recovers the product formula \cite[Theorem 6.4]{EHKR} for the Poincar\'e series of $M_n$. 





\subsection*{Acknowledgments.}
The authors would like to thank B. Farb, J. Mostovoy, B. Cisneros and J. Wilson for useful conversations and comments. We are grateful to the referee for all the suggestions that  helped  improve the exposition in this paper.

\section{The space $M_n$ and its cohomology ring }\label{REALMODULI}


The Deligne-Mumford compactification $\overline{\mathcal{M}}_{0,n}$ of the moduli space $\mathcal{M}_{0,n}$ of rational curves with $n$ marked points is a smooth projective variety over $\mathbb{Q}$ that parametrizes stable curves of genus $0$ with $n$ labeled points (see \cite{DM}). 

\begin{definition} A \textit{stable rational curve} with $n$ marked points is a finite union $C$ of projective lines $C_1,C_2,\ldots, C_k$ together with distinct marked points $z_1,\dots,z_n$ in $C$ such that
\begin{itemize}
\item each marked point $z_i$ is in only one $C_j$;
\item $C_i\cap C_j$ is either empty or consists only of one point, and in this case the intersection is transversal;
\item its associated graph, with a vertex for each $C_i$ and an edge for each pair of interesection lines,  is a tree;
\item each component $C_i$ has at least three special points, where a special point is either an intersection point
with another component, or a marked point.
\end{itemize}
\end{definition}

In $\overline{\mathcal{M}}_{0,n}$ we say that  two stable curves $C=\{C_1,\ldots,C_k,(z_1,\ldots, z_n)\}$ and $C'=\{C'_1,\ldots,C'_k,(z'_1,\ldots, z_n')\}$ are equivalent if there is an isomorphism of algebraic curves $f:C\rightarrow C'$ such that $f(z_i)=z'_i$. Notice  $f$ is given by a collection of $k$ fractional linear maps $f_j:C_j\rightarrow C'_{\sigma(j)}$, where $\sigma\in S_k$. In particular, if $f$ is an automorphism of $C$, each $f_i$ is an automorphism of $\mathbb{P}^1$ that fixes  at least three points (the special points of the component $C_i$).  Hence a stable curve $C$ has no non-trivial automorphisms.

Let $n,m\geq 3$ and consider an injection  $f:[m]\hookrightarrow [n]$, of the set $[m]:=\{1,\ldots,m\}$ into  $[n]:=\{1,\ldots,n\}$. There is a well-defined ``forgetful morphism"   $$\phi_f:\overline{\mathcal{M}}_{0,n}\rightarrow \overline{\mathcal{M}}_{0,m}$$ given by $\phi_f (C)=\phi_f\big(\{C_1,\ldots,C_k,(z_1,\ldots, z_n)\}\big):=\{C'_1,\ldots,C'_r,(z_{f(1)},\ldots, z_{f(m)})\}$ where the components $C'_i$ of $\phi_f (C)$ are the same as the components $C_i$ of $C$ unless they become unstable (see figure below). Notice that $\phi_f$  forgets all the marked points whose index are not in the image $f([m])$. If after this ``forgetting process'' the component $C_i$ has less than three special points, then we contract the component $C_i$ so that $\phi_f(C)$ is still a stable curve (for more details see for example \cite[Chapter 1]{QUANTUM} and references therein). Furthermore, by taking $f\in$End$([n])$, we obtain a natural action of the symmetric group $S_n$ on $\overline{\mathcal{M}}_{0,n}$ by permuting the marked points.

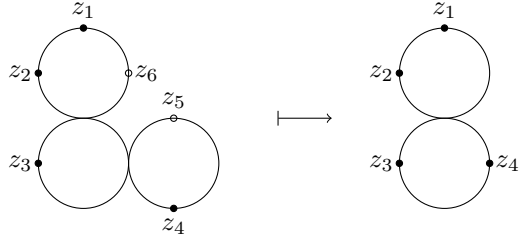
\begin{figure}[H]

\caption{ Forgetful morphism for the canonical inclusion  $[4]\hookrightarrow [6].$} 

$$
\begin{tikzpicture}[scale = 0.6]

\draw (0,0) circle (1cm);
\draw (0,2) circle (1cm);
\draw (2,0) circle (1cm);
\draw (8,0) circle (1cm);
\draw (8,2) circle (1cm);

\draw (-1.4,0) node  {$z_3$};
\filldraw [black] (-1,0) circle (2pt);

\draw (-1.4,2) node  {$z_2$};
\filldraw [black] (-1,2) circle (2pt);

\draw (0,3.4) node  {$z_1$};
\filldraw [black] (0,3) circle (2pt);

\draw (1.4,2) node  {$z_6$};
\draw [black] (1,2) circle (2pt);

\draw (2,1.4) node  {$z_5$};
\draw [black] (2,1) circle (2pt);

\draw (2,-1.4) node  {$z_4$};
\filldraw [black] (2,-1) circle (2pt);

\draw (8,3.4) node (1) {$z_1$};
\filldraw [black] (8,3) circle (2pt);

\draw (6.6,2) node (1) {$z_2$};
\filldraw [black] (7,2) circle (2pt);

\draw (6.6,0) node (1) {$z_3$};
\filldraw [black] (7,0) circle (2pt);

\draw (9.4,0) node (1) {$z_4$};
\filldraw [black] (9,0) circle (2pt);

\draw [|->](4.3,1)--(5.5,1);

\end{tikzpicture}
$$
\end{figure}

Our space of interest is the real locus of this variety $M_n:=\overline{\mathcal{M}}_{0,n}(\mathbb{R})$. It can be described as the set of equivalence classes of stable curves with $n$ marked points defined over $\mathbb{R}$. Since the projective lines are circles over the real numbers, a stable curve is a tree of circles with labeled points on them (as referred to  in \cite{EHKR} as a ``cactus-like'' structure). 
The space $M_n$ is a compact and connected smooth manifold of dimension $n-3$ (see \cite{DJS} and \cite{DEV}). It is clear that $M_3$ is a point and the cross ratio give us an isomorphism between 
$M_4$ and the circle. It can also be shown that $M_5$ is a connected sum
of five real projective planes (see \cite{DEV}).

The cohomology ring of $M_n$ was completely determined in \cite{EHKR}. We recall their description next.

\begin{theorem}[\cite{EHKR} 2.9]\label{thm:cohomology ring} Let $\Lambda_n$ be the skew-commutative algebra generated by the elements of degree one $w_{ijkl}$, $1\leq i,j,k,l\leq n$, which are antisymmetric  in $i,j,k,l$\footnote{Meaning $-w_{ijkl}=w_{jikl}=w_{ikjl}=w_{ijlk}$.}  with defining relations
$$w_{ijkl}+w_{jklm}+w_{klmi}+w_{lmij}+w_{mijk}=0$$
$$w_{ijkl}w_{ijkm}=0.$$
Then $\Lambda_n$ is isomorphic to $H^*(M_n,\mathbb{Q})$, and the action of $S_n$ is given by 
$$\pi(w_{ijkl})=w_{\pi(i)\pi(j)\pi(k)\pi(l)}.$$
\end{theorem}	

As we mentioned before, for $n,m\geq 3$ and each $m$-subset $S$ of $[n]$, there is a well-defined forgetful map  $\phi_S:M_n\rightarrow M_m$ which induces a map in cohomology
$\phi_S^*:H^*(M_m,\mathbb{Q})\rightarrow H^*(M_n,\mathbb{Q}).$ If $m=4$ and $S=\{i,j,k,l\}$ with $1\leq i<j<k<l\leq n$, the induced map  $\phi_S^*: H^*(M_4,\mathbb{Q})\rightarrow H^*(M_n,\mathbb{Q})$  takes the generator of  $H^*(M_4,\mathbb{Q})\cong\mathbb{Q}$  to $w_{ijkl}$.

Observe that since  $w_{ijkl}=w_{1jkl}-w_{1kli}+w_{1lij}-w_{1ijk}$, the set $${\{w_{1ijk}|1<i<j<k\leq n\}}$$ is a generating set for the algebra $\Lambda_n$. Also note that since there are not other linear relations that set is also a basis for the vector space $H^1(M_n,\mathbb{Q})$, we will use this fact below.

\section{Representation stability and $FI$-modules }\label{FIREP}

In this section we revisit what it means for a sequence of $S_n$-representations  to be uniformily representation stable in the sense of \cite{CF}
and we survey the main notions  from the  theory of $FI$-modules developed by Church, Ellenberg and Farb in \cite{CEF} that can be used to prove such phenomenon. Our main purpose is to discuss how these ideas can be readily applied to certain sequences of groups or spaces (see examples in Section \ref{OTROS})  with cohomology rings that are known to be generated by the first cohomology group.  This approach was first introduced in \cite[Section 4.2]{CEF} for a more general setting and was applied to another collection of examples in \cite[Section 5]{CEF}.
We  illustrate the concepts with the example given by the cohomology of the moduli spaces $M_n$ and obtain the proof of ~\autoref{thm:main}. 

\subsection{Representation stability.}
Church and Farb introduced in \cite{CF} a notion of stability for certain sequences of $S_n$- representations over a field $k$ of characteristic zero. We recall their definitions here.

\begin{definition}
A sequence  $\{V_n, \phi_n\}_n$ of $S_n$-representations over $k$ together with homomorphisms $\phi_n :V_n\rightarrow V_{n+1}$ is said to be {\it consistent} if the maps $\phi_n$ are $S_n$-equivariant with respect to the natural inclusion $S_n\hookrightarrow S_{n+1}$.


\end{definition}

\noindent {\bf Notation:} 
Over a field of characteristic zero, irreducible representations of $S_n$ are defined over $\mathbb{Q}$ and every $S_n$-representation decomposes as a direct sum of irreducible representations. Furthermore, irreducible representations of $S_n$ are classified  by partitions of $n$. By a partition of $n$ we mean a collection $\lambda=(l_1,\ldots,l_r)$ with ${|\lambda|:=l_1+\cdots +l_r=n}$ and $l_i\geq l_{i+1}>0$. We will denote such a partition by $\lambda\vdash n$ and the associated irreducible representation by $V_\lambda$.
 Following the notation in \cite{CEF}, given a partition $\lambda=(l_1,\ldots,l_j)\vdash m$ of a positive integer $m$ and an integer $n\geq m+l_1$, we define the $S_n$ representation $V(\lambda)_n$ as 
$$V(\lambda)_n:=V_{\lambda[n]},$$ 
the irreducible representation that corresponds to the padded partition  $\lambda[n]:=(n-m,l_1,\ldots,\l_j)$ of $n$.

\begin{definition}\label{def:rep stab}(Uniform representation stability) 
Let $\{V_n,\phi_n\}_n$ be a consistent sequence of $S_n$-representations. This sequence is {\it uniformly representation stable}  if there exists a natural number $N$ such that for $n\geq N$ the following conditions are satisfied:

\begin{enumerate}
\item  The map $\phi_n$ is injective.
\item  The span of the $S_{n+1}$-orbit of $\phi_n(V_n)$ is equal to $V_{n+1}$.
\item If $V_n$ is decomposed into irreducible representations as $\bigoplus_\lambda c_{\lambda,n}V(\lambda)_n$, 
 the multiplicities $c_{\lambda, n}$ for each $\lambda$ are independent of $n$.
 \end{enumerate}
\end{definition}

\subsection{\bf The $FI$-module structure.} It was noticed in \cite{CEF} that certain consistent sequences of $S_n$-representations could be encoded as a single object and this perspective had strong advantages. Using their notation, we consider the category $FI$ whose objects are all finite sets and whose morphisms are all injections. 
\begin{definition}
An {\it $FI$-module} over a commutative ring $R$ is a functor from the category $FI$ to the category of $R$-modules. If $V$ is an {\it$FI$-module} we denote by $V_n$ or $V(\bf{n})$ the $R$-module associated to $\mathbf{n} :=[n]$ and by $f_*:V_m\rightarrow V_n$ the map corresponding to the inclusion $f\in\text{Hom}_{FI}(\mathbf{m},\mathbf{n})$.
\end{definition}

The category of $FI$-modules over $R$ is closed under covariant functorial constructions on $R$-modules, if we apply functors pointwise. In particular, if $V$ and $W$ are $FI$-modules, then  $V \otimes W$ and $V \oplus W$ are $FI$-modules.

Notice that since End$_{FI}$($\bf{n}$)= $S_n$, an $FI$-module  $V$ encodes the information of the consistent sequence $\{V_n, (I_n)_*\}$ of $S_n$-representations with the maps induced by the natural inclusion $I_n:[n]\hookrightarrow [n+1]$. It should be noted that not every consistent sequence can be encoded in this way \cite[Remark 3.3.1]{CEF}. \medskip 

 \noindent{\bf The $FI$-modules $H^i(M_\bullet,\mathbb{Q})$.} In this paper we are interested in certain sequences of spaces or groups that can be encoded as a contravariant functor $Y_\bullet$  from the category $FI$ to the category of spaces, a {\it co-FI-space}, or to the category of groups, a {\it co-FI-group}. Therefore,  by composing with the contravariant functors $H^i(-,R)$ and $H^*(-,R)$, we obtain the $FI$-modules $H^i(Y_\bullet,R)$ over $R$   for each $i\geq 0$ and $H^*(Y_\bullet,R)$ which is a functor from $FI$ to the category of graded $R$-algebras, what \cite{CEF} calls a {\it graded FI-algebra} over $R$, . 

In particular, we can think of the sequence of manifolds $\{M_n\}$ as the objects of the  co-$FI$-space $M_\bullet$ that takes each $\mathbf{n}$ to $M_n$ and each inclusion $f\in\text{Hom}_{FI}(\mathbf{m},\mathbf{n})$ to the corresponding forgetful map $\phi_f: M_n\rightarrow M_m$ defined as before. In this section, for each $i\geq0$ we sometimes denote the $FI$-module $H^i(M_\bullet,\mathbb{Q})$ simply by $H^i$ and the graded $FI$-algebra $H^*(M_\bullet,\mathbb{Q})$ by $H^*$.

\subsection{\bf Finite generation of an $FI$-module.}\label{FINITE} 
We are not interested in all $FI$-modules, but in those that are finitely generated in the following sense.	

\begin{definition}

\begin{enumerate}[a)]

 Let $V$ be an $FI$-module over $R$.
\item If  $\Sigma$ is a subset of the disjoint union $\bigsqcup_n V_n$ we define  span$_V(\Sigma)$, the {\it span of $\Sigma$}, to be the minimal sub-$FI$-module of $V$ containing $\Sigma$ and we say that $\Sigma$ {\it generates} the FI-module  span$_V(\Sigma)$.
\item We say that $V$ is {\it generated in degree $\leq m$}  if span${_V\big(\bigsqcup_{k\leq m}V_k\big)=V}$.
\item We say that $V$ is {\it finitely generated} if there is a finite set of elements $\{v_1,\ldots, v_k\}$ with $v_i\in V_{n_i}$ which generates $V$, that is, span${_V(\{v_1,\ldots,v_k\})=V}$.

\end{enumerate}
\end{definition}

Finite generation of $FI$-modules is a property that is closed under quotients and extensions. It also  passes to sub-$FI$-modules when $R$ is a field that contains $\mathbb{Q}$ (see \cite[Theorem 1.3]{CEF}) and more generally when $R$ is a Noetherian ring (see \cite[Theorem A]{CEFN}). It is key that finite generation is closed under tensor products and we have control of the degree of generation.

\begin{proposition}\label{TENSOR}
\cite[Prop. 2.3.6]{CEF}  
If $V$ and $W$ are finitely generated $FI$-modules, so is $V \otimes W$. If $V$ is generated in degree $\leq ≤ m_1$ and $W$ is generated in degree $\leq m_2$, then
$V \otimes W$ is generated in degree $\leq m_1+m_2$.
\end{proposition}

Our examples below are graded $FI$-algebras and we will take advantage of this structure to obtain finite generation.

\begin{definition}
Let $A$ be a graded $FI$-algebra over $R$. Given a graded sub-$FI$-module $V$ of $A$, we say that {\it $A$ is generated by $V$} if $V_n$ generates $A_n$ as an $R$-algebra for all $n\geq 0$.
\end{definition}

From ~\autoref{thm:cohomology ring} , the graded $FI$-algebra $H^*(M_\bullet,R)$ is generated by the $FI$-module $H^1(M_\bullet,R)$ (which can be thought as the graded sub-$FI$-module concentrated in grading $1$). In all the examples in Section \ref{OTROS}, the graded $FI$-algebra of interest $H^*(Y_\bullet,R)$ is  a quotient of the free tensor algebra on the generating classes in $H^1(Y_\bullet,R)$.  Therefore, Proposition \ref{TENSOR} reduces the question of finite generation for the $FI$-modules $H^i(Y_\bullet,R)$ to a question of finite generation for the $FI$-module $H^1(Y_\bullet,R)$. 

\begin{proposition}\label{GENBYH1}
Suppose that the $FI$-module $H^1(Y_\bullet,R)$ is finitely generated in degree $\leq g$. If the graded $FI$-algebra $H^*(Y_\bullet,R)$ is generated by $H^1(Y_\bullet,R)$, then for each $i\geq 1$ the FI-module $H^i(Y_\bullet,R)$ is finitely generated in degree $\leq g\cdot i$.
\end{proposition}

\noindent {\bf Degree of generation of $H^i(M_\bullet,\mathbb{Q})$}. Using the explicit description of the cohomology ring in ~\autoref{thm:cohomology ring} we can obtain finite generation for the $FI$-modules $H^i$ from Proposition \ref{GENBYH1}.  Moreover, we can obtain the following  upper bound for the degree of generation.\\

Notice that, in particular, $H^1$ is finitely generated in degree $\leq 4$ and  Proposition \ref{GENBYH1} only implies that $H^{i}$ is finitely generated in degree $\leq 4i$.\\
\begin{lemma}\label{lemmadeg}
The $FI$-module $H^i$ is finitely generated in degree $\leq 3i+1$.
\end{lemma}
\begin{proof}
Recall that $H^1(M_n,\mathbb{Q})$ is generated by elements $w_{1klm}$ with  $k,l,m$ between $2$ and $n$, and also that it generates the whole cohomology ring. Hence the vector space $H^i(M_n,\mathbb{Q})$ is generated by elements $$x=w_{1k_1l_1m_1}\cdots w_{1k_il_im_i}$$ with the indices varying from 2 to $n$. 
 A generator $\mbox{element}$ $x$ has at most $3i+1$ different indices in the set $[n]$. Then, $\mbox{after}$ the $\mbox{action}$ of some permutation $\sigma$ of $S_n$,  $\sigma\cdot x$  is in the image of $H^i(M_{3i+1},\mathbb{Q})$ inside $H^i(M_n,\mathbb{Q})$. Hence 
$${x=\sigma^{-1}(\sigma\cdot x)}  \in \mbox{span}_{H^i}(H^i_{3i+1}),$$ therefore $H^{i}=$ span$_{H^{i}}(H^{i}_{3i+1})$. Since $H^{i}(M_{3i+1},\mathbb{Q})$ is of finite dimension, we conclude that the $FI$-module $H^{i}$ is finitely generated in degree $\leq 3i+1$. 
\end{proof}

One of the highlights of the theory of $FI$-modules is that finite generation of an $FI$-module over a field $k$ of characteristic zero is equivalent to uniform representation stability of the corresponding consistent sequence.



\begin{theorem}\label{thm:FI}\cite[Thm 1.13 and Prop. 3.3.3]{CEF}
An $FI$-module over $k$ is finitely generated if and only if the sequence $\{V_n,(I_n)_* \}$ of $S_n$-representations is uniformly representation stable and each $V_n$ is finite dimensional. Furthermore, the stable range is $n\geq s+d$ where $d$ and $s$ are the weight and the stability degree of $V$ respectively.
\end{theorem}

In Sections \ref{WEIGHT} and \ref{STAB} below we recall the definitions of weight and  stability degree of an $FI$-module. We focus on $FI$-modules over a field $k$ of characteristic zero.\\

\noindent {\bf Representation stability for the cohomology of $M_{n}$.} We just proved finite generation for the $FI$-module $H^i$ in Lemma \ref{lemmadeg}.  Therefore, the theorem above implies uniform representation stability for the sequence $\big\{H^i(M_n,\mathbb{Q})\big\}$. The specific stable range in ~\autoref{thm:main} follows from the computations of weight and stability degree of the $FI$-module $H^i$ in Lemmas \ref{lemmaweight} and \ref{lemmastab} below.

\begin{remark} Representation stability for the cohomology of the pure cactus groups can also be obtained by the methods in \cite{CF} where the notion was first introduced. This approach, however, allows us to prove only the stability of $H^i(M_n,\mathbb{Q})$ with $\mbox{respect}$ to the $S_{n-1}$-action (rather than the action of $S_n$) and we obtain a worse estimate for the stable range than the one achieved using $FI$-modules.
\end{remark}

\subsection{Weight of an FI-module}\label{WEIGHT}

It turns out that finite generation of an $FI$-module $V$ puts certain constraints on the partitions (shape of the Young diagrams) in the irreducible representations of each representation $V_n$ as we discuss below.

\begin{definition}
Let $V$ be an $FI$-module over $k$. We say that $V$ has {\it weight} $d$ if, for all $n\geq 0$, $d$ is the maximum order $|\lambda|$ over all the irreducible constituents $V(\lambda)_n$ of $V_n$. We write weight$(V)=d$.
\end{definition}

By definition, if $W$ is a subquotient of $V$ , then weight$(W)$ $\leq$ weight$(V)$. Moreover, the degree of generation of an $FI$-module gives an upper bound for the weight.

\begin{proposition}\cite[Prop. 3.2.5]{CEF}\label{DEGBOUNDSW} If the $FI$-module $V$ over $k$ is generated in degree $\leq g$, then weight$(V)\leq g$.
\end{proposition}

This implies, for example,  that for a finitely generated $FI$-module $V$ the alternating representation  $V (\lambda)_n$, where $\lambda = (1,\ldots, 1)$ has $|\lambda|=n-1$,  cannot appear in the decomposition into irreducibles for $n\gg 0$. We also have control of the weight under tensor products.
\begin{proposition}\cite[Prop. 3.2.2]{CEF} If $V$ and $W$ are $FI$-modules over $k$, then weight$(V \otimes W) \leq$ weight$(V ) + $weight$(W)$.
 \end{proposition}

As a consequence we have the following result that can be applied to our examples of interest.

\begin{corollary}\label{WEIGHTBYH1}
If the graded $FI$-algebra $H^*(Y_\bullet,k)$ is generated by the $FI$-module $H^1(Y_\bullet,k)$ and weight$\big(H^1(Y_\bullet,k)\big)\leq d$, then for each $i\geq 1$ we have that weight$\big(H^i(Y_\bullet,k)\big)\leq d\cdot i$.
\end{corollary}

\noindent{\bf The weight of $H^i(M_\bullet,\mathbb{Q})$.} From Proposition \ref{DEGBOUNDSW} and our computation above we have that the weight of $H^i(M_\bullet,\mathbb{Q})$ is bounded above by $3i+1$. We can do slightly better. 

Let $\mathcal{H}_n$ be the standard representation of $S_n$, the subrepresentation of dimension $n-1$ of the permutation representation $\mathbb{Q}^n$ consisting of vectors with coordinates that add to zero. We consider the map $H^1(M_n,\mathbb{Q})\rightarrow \bigwedge^3 \mathcal{H}_n$ given by $w_{ijkl}\mapsto (e_i-e_l)\wedge(e_j-e_l)\wedge(e_k-e_l)$ and observe that it respects the antisymmetry and the five terms relation from  the description in ~\autoref{thm:cohomology ring}. Moreover, the map is clearly surjective and a dimension count shows that it is an isomorphism of $S_n$-representations as pointed out in \cite[Prop. 2.2]{EHKR}.

  The $S_n$-representation $\bigwedge^3 \mathcal{H}_n$ is irreducible and corresponds to the partition $(n-3,1,1,1)$ of $n$. With our notation above this means that $H^1(M_n, \mathbb{Q})=V(1,1,1)$ for $n\geq 4$ and weight$\big(H^1(M_\bullet,\mathbb{Q})\big)=3$. From Proposition \ref{WEIGHTBYH1} we obtain that weight$\big(H^i(M_\bullet,\mathbb{Q})\big)\leq 3i$. Hence we have shown.

\begin{lemma}\label{lemmaweight}  The $FI$-module $H^i$ has weight $\leq 3i$.
\end{lemma}

\subsection{Stability degree of an FI-module}\label{STAB}
In order to define the stability degree we recall the notion of coinvariants of a representation.


 Let $V_n$ be an $S_n$-representation.  The coinvariants $(V_n)_{S_n}$ is the largest $S_n$-equivariant quotient of $V_n$ on which $S_n$ acts trivially. Equivalently, it is the module $V_n\otimes_{k[S_n]} k$.  If V is an $FI$-module, we can apply this construction to all the $V_n$ simultaneously, obtaining the $k$-modules $(V_n)_{S_n}$ for each $n\geq0$. Moreover, all the maps $V_n\to V_{n+1}$ involved in
the definition of the $FI$-module structure induce a single map $T: (V_n)_{S_n}\rightarrow (V_{n+1})_{S_{n+1}}$. 


 For each integer $a\geq 0$ fix once and for all some set $\bar{\mathrm{\bf{a}}}$ of cardinality $a$, for instance, take $\bar{\mathrm{\bf{a}}}=\{-1,\ldots,-a\}$. We consider the action of $S_n$ in $V_{\bar{\mathrm{\bf{a}}}\sqcup\bf{n}}$ induced by the elements in End$_{FI}(\bar{\mathrm{\bf{a}}}\sqcup\bf{n})$  that are the identity on $\bar{\mathrm{\bf{a}}}$ and take the coinvariants $(V_{\bar{\mathrm{\bf{a}}}\sqcup\bf{n}})_{S_n}$.


 \begin{definition}
 The {\it stability degree} stab-deg($V$) of an $FI$-module $V$ is the smallest $s\geq 0$ such that for all $a\geq 0$, the map $T:(V_{\bar{\mathrm{\bf{a}}}\sqcup\bf{n}})_{S_n}\rightarrow (V_{\bar{\mathrm{\bf{a}}}\sqcup\bf{n+1}})_{S_{n+1}}$, induced by any inclusion $[n]\hookrightarrow[n+1]$, is an isomorphism for all $n\geq s$.
 \end{definition}

For $FI$-modules over $\mathbb{Q}$, when taking tensor products, the stability degree is still bounded above.

\begin{proposition}\cite[Prop. 2.23]{KMILLER}
If the $FI$-modules $V$ and $W$ over $\mathbb{Q}$  have stability degree $\leq s_1,s_2$ and weight $\leq d_1, d_2$ respectively, then $V\otimes W$ has stability
degree $\leq \max\{s_1 + d_1, s_2 + d_2, d_1 + d_2\}$.
\end{proposition}

This implies the following result in the setting of interest for this paper.

\begin{corollary}\label{STABDEGBYH1}
If the graded $FI$-algebra $H^*(Y_\bullet,\mathbb{Q})$ is generated by the $FI$-module $H^1(Y_\bullet,\mathbb{Q})$ with weight$\leq d$ and stability degree $\leq s$, then for each $i\geq 1$ we have that stab-deg$\big(H^i(Y_\bullet,\mathbb{Q})\big)\leq \max\{s+d,d\cdot i\}$.
\end{corollary} 

\noindent{\bf Stability degree of $H^i(M_\bullet,\mathbb{Q})$.} We now find an upper bound for the stability degree for our example of interest.

\begin{lemma}\label{lemmastab} The $FI$-module $H^i$ has stability degree $\leq 3i$.
\end{lemma}

\begin{proof}
First of all we describe the vector space $H^{1}(M_{|\bar{\mathrm{\bf{a}}}\sqcup \bf{n}|},\mathbb{Q})$ and the action of $S_n$ on it. The generating elements of this vector space are of the form $w_{ijkl}$ where the indexes are now in the set $\bar{\mathrm{\bf{a}}}\sqcup[n]$ and the only linear relations among them are the  usual five term relations, hence if $n\geq1$ the set $\{w_{1jkl}|j,k,l\in \bar{\mathrm{\bf{a}}}\sqcup[n]\}$ is a generating set for this vector space and $S_n$ acts only in the indexes in $[n]$.\\
Given $a\geq0$, denote by
 $${(V_{\bar{\mathrm{\bf{a}}}\sqcup\bf{n}})_n=\big(H^{1}(M_{|\bar{\mathrm{\bf{a}}}\sqcup \bf{n}|},\mathbb{Q})\big)_{S_n}}$$
 
the vector space of coinvariants of $H^{1}(M_{|\bar{\mathrm{\bf{a}}}\sqcup \bf{n}|},\mathbb{Q})$ where $S_n$ acts only in the indices in $[n]=\{1,\ldots,n\}$. Therefore  $(V_{\bar{\bf{a}}\sqcup\bf{n}})_n= H^{1}(M_{|\bar{\mathrm{\bf{a}}}\sqcup \bf{n}|},\mathbb{Q})/W$, where $W$ is the subspace generated by $\{v-\sigma\cdot v|\  v\in H^{1}(M_{|\bar{\mathrm{\bf{a}}}\sqcup \bf{n}|},\mathbb{Q})\text{ and }\sigma\in S_n\}$. We claim that  ${B=\{w_{1jkl}|\mbox{ at least one } j,k,l\in [n]\}}$ is a generating set for $W$. First of all notice that if $j\in[n]$ we have
$$w_{1jkl}=\frac{1}{2}(w_{1jkl}-(1\ j)\cdot w_{1jkl})$$
and $w_{1jkl}\in W$, similarly if $k$ or $l$ belongs to $[n]$.
On the other hand, note that if $j,k,l$ are not in $[n]$  and $\sigma\in S_n$ we have 
$$\sigma\cdot w_{1jkl}=w_{\sigma(1)jkl}=w_{1jkl}-w_{kl1\sigma(1)}-w_{l1\sigma(1)j}-w_{1\sigma(1)jk},$$
then the difference $w_{1jkl}-\sigma\cdot w_{1jkl}$ is in the span of $B$.  Moreover, notice that  if $w_{1jkl} \in B$, its image by any permutation is in $W$. Therefore  $B$  spans $W$ and it is a basis, up to the anti-symmetry of the indexes, for  $W$. Hence, for $n\geq 1$, $(V_{\bar{\mathrm{\bf{a}}}\sqcup\bf{n}})_n$ has $\{w_{1jkl}|\  j<k<l\mbox{ and } j,k,l \in \bar{\mathrm{\bf{a}}}\}$ as basis. Observe that this basis does not depend on $n$.  Since the map
$$T:(V_{\bar{\mathrm{\bf{a}}}\sqcup\bf{n}})_n\rightarrow (V_{\bar{\mathrm{\bf{a}}}\sqcup\bf{n+1}})_{n+1}$$   takes the basis to itself, it is an isomorphism for $n\geq 1$.
	
Therefore stab-deg($H^{1}$)$\leq 1$. If $|\bar{\mathrm{\bf{a}}}|=3$ we have $(V_{\bar{\mathrm{\bf{a}}}\sqcup\bf{0}})_0=H^1(M_3,\mathbb{Q})=0$ and $(V_{\bar{\mathrm{\bf{a}}}\sqcup\bf{1}})_1=H^1(M_4,\mathbb{Q})=\mathbb{Q}$ thus stab-deg($H^1$)$\neq0$ therefore stab-deg($H^{1}$)=$1$. Since we know that weight$(H^1)=3$, the upper bound for the stability degree of $H^i$ follows from Corollary \ref{STABDEGBYH1}.

\end{proof}


\subsection{Polynomiality of characters.}
We end this section by discussing another way that finite generation of $FI$-modules imposes strong constraints in the corresponding $S_n$-representations. 
If $V$ is finitely generated, 
then  the sequence of characters $\chi_{V_n}$ of the $S_n$-representations $V_n$ is {\it eventually polynomial} as we now explain.
For each $i \geq 1$, we consider the class function $X_l:\bigsqcup_n S_n\rightarrow \mathbb{Z}$  defined by
$$X_l(\sigma) = \text{ number of $l$-cycles in the cycle decomposition of }\sigma.$$
As mentioned in the Introduction, polynomials in the variables $X_l$ are called {\it character polynomials}. 
The degree of the character polynomial $P(X_1,\ldots X_r)$ is defined by setting $\deg(X_l) = l$.

\begin{theorem}[\cite{CEF}Theorem 3.3.4]\label{thm:CHAR} Let $V$ be an $FI$-module over $\mathbb{Q}$. If $V$ is finitely generated then there exist a unique character polynomial $P_V\in\mathbb{Q}[X_1,X_2,\ldots,X_d]$ with $\deg P_V\leq d$ such that for all $n\geq s+d$ 
$$\chi_{V_n}(\sigma) = P_V(\sigma)\text{ for all }\sigma\in S_n,$$
where $d$ and $s$ are the weight and the stability degree of $V$ respectively.
\end{theorem}
The upper bound on the degree of the character polynomial $P_V$ in ~\autoref{thm:CHAR}  implies that $P_V$ only involves the variables $X_1,\ldots,X_d$. This tells us that the character $\chi_{V_n}$ only depends on ``short cycles"', i.e. on cycles of length $\leq d$, independently of how large $n$ is. Moreover, we have that the dimension of $V_n$
$$\dim V_n=\chi_{V_n}(\text{id})=P_V(n,0,\ldots,0)$$
 is a polynomial in $n$ of degree $\leq d$ for $n\geq s+d$.\\

\noindent {\bf Characters for the cohomology of $M_{n}$.} ~\autoref{thm:CHAR} together with the computations below of upper bounds for the weight and the stability degree for the $FI$-modules $H^i$ give us Theorem \ref{thm:CHARCACTUS}. 
Corollary \ref{BETTICACTUS} follows from this theorem and  the discussion before.



\section{Other examples}\label{OTROS}
We would like to end this note by revisiting the collection of ``pure braid like'' examples in Table \ref{EXA} that have been considered using the perspective discussed in Section \ref{FIREP}. In general, the bounds in Table \ref{EXA} are not sharp.

\begin{landscape}
\begin{center}
 \begin{table}[h]
\begin{center} 
   \caption{Examples of  ``pure braid like" goups and spaces that have cohomology rings with an FI-algebra structure generated by the first cohomology. } \label{EXA}\end{center}
 \centering\small
        \hspace*{-1cm}\begin{tabular}{  |c|c|c|c|c|c|c|  }
					\hline
					 &  & & & & &\\
					
{\bf Group or}  &{\bf Cohomology Ring }& & & weight of   & deg-gen of& stab-deg of \\
{\bf Space $Y_n$}  &{\bf $H^*(Y_n,\mathbb{Q})$}& $H^1(Y_n,\mathbb{Q})$ &  $\chi_{H^1(Y_n,\mathbb{Q})}$ &$H^p(Y_\bullet)\leq$&$H^p(Y_\bullet)\leq$&$H^p(Y_\bullet)\leq$\\
			 & & & &  & &\\
           \hline
	
&  Exterior algebra $\mathcal{R}_n$ generated by  & & & & &\\
$\mathcal{F}(\mathbb{C},n)$ & $w_{i,j}$, $1\leq i\neq j\leq n$, with relations $w_{i,j}=w_{j,i}$,& $V(\cdot)\oplus V(1)\oplus V(2)$ & $\binom {X_1}{2}+X_2$ & $2p$ & $2p$ & $2p$\\ 
&  $w_{i,j}w_{j,k}+w_{j,k}w_{k,i}+w_{k,i}w_{i,j}=0$& for $n\geq 4$ & for $n\geq 0$ & & & \\
  
 &\begin{footnotesize}{\it (See \cite{ARNOLD2})}\end{footnotesize} &   & &  & &\begin{footnotesize}{\it (See \cite{CEF})}\end{footnotesize} \\
\hline
& Subalgebra of $\mathcal{R}_n$ generated by  & & & & &\\
$\mathcal{M}_{0,n+1}$ &$1$ and $\theta_{i,j}:=w_{i,j}-w_{1,2}$ with $\{i,j\}\neq\{1,2\}$ & $V(1)\oplus V(2)$ & $\binom {X_1}{2}+X_2 -1$ & $2p$ & $4p$ & $2p$\\
 &\begin{footnotesize}{\it (See \cite[Cor. 3.1]{GAIFFI} and references therein)}\end{footnotesize} & for $n\geq 4$ & for $n\geq 2$ & & & \begin{footnotesize}{\it (See \cite{JIM3})}\end{footnotesize}\\

\hline

 & Supercommutative algebra generated by & & & & &  \\
 & $w_{ijkl}$   $1\leq i,j,k,l\leq n$, antisymmetric in & &$\binom{X_1}{3}+X_3-X_2X_1$ & & &  \\
$\overline{\mathcal{M}}_{0,n}(\mathbb{R})$ & $i,j,k,l$ with relations $w_{ijkl}w_{ijkm}=0$ and & $V(1,1,1)$ & $-\binom{X_1}{2}+X_2+X_1-1$ & $3p$ & $3p +1$ & $3p$  \\
 &$w_{ijkl}+w_{jklm}+w_{klmi}+w_{lmij}+w_{mijk}=0$& for $n\geq 4$ & for $n\geq 3$ & & &  \\
 &\begin{footnotesize}{\it (\cite[Thm. 2.9]{EHKR})}\end{footnotesize}& & & \begin{footnotesize}{\it (\ref{lemmaweight})}\end{footnotesize}& \begin{footnotesize}{\it (\ref{lemmadeg})}\end{footnotesize}&   \begin{footnotesize}{\it (\ref{lemmastab})}\end{footnotesize}\\
\hline
 &Exterior algebra  generated by $w_{i,j}$, & & & & &\\
$PvB_n$ &  $1\leq i\neq j\leq n$,  with relations $w_{i,j}w_{j,i}=0$,& $V(\cdot)\oplus V(1)^2\oplus$  & $2\binom{X_1}{2}$ & $2p$ & $2p$ & $2p$\\ 
& $w_{i,j}w_{i,k}=w_{i,j}w_{j,k}-w_{i,k}w_{k,j}$ and & $V(1,1)\oplus V(2)$& for $n\geq 0$ & & &  \\
& $w_{i,k}w_{j,k}=w_{i,j}w_{j,k}-w_{j,i}w_{i,k}$ &for $n\geq 4$  & & & &  \\
 &\begin{footnotesize}{\it (See \cite[Section 5]{LEE1}) and \cite{BARTH}}\end{footnotesize}    & &  & \begin{footnotesize}{\it (See \cite{LEE2})}\end{footnotesize} & & \\
\hline
 &Exterior algebra  generated by $w_{i,j}$   & & & & &\\
$Pf	B_n$& $1\leq i\neq j\leq n$, with relations & $V(1) \oplus V(1,1)$ & $\binom{X_1}{2}-X_2$ & $2p$ & $2p$ & $2p$\\ 
 &$w_{i,j}=-w_{j,i}$, $w_{i,j}w_{i,k}=w_{i,j}w_{j,k}$ and& for $n\geq 3$ & for $n\geq 0$ & & &  \\
 &$w_{i,k}w_{j,k}=w_{i,j}w_{j,k}$ for $i<j<k$&  &  & & &  \\
 &\begin{footnotesize}{\it (See \cite[Section 5]{LEE1}) and \cite{BARTH}}\end{footnotesize}   &  &  & \begin{footnotesize}{\it (See \cite{LEE2})}\end{footnotesize}& & \\
\hline
 &Exterior algebra  generated by $w_{i,j}$ & & & & &\\
$P\Sigma_n$ & $1\leq i\neq j\leq n$, with relations $w_{i,j}w_{j,i}=0$,&  $V(\cdot)\oplus V(1)^2\oplus$  & $2\binom{X_1}{2}$ & $2p$ & $2p$ & $2p$\\ 
&and $w_{k,j}w_{j,i}-w_{k,j}w_{k,i}+w_{i,j}w_{k,i}=0$ &$V(1,1)\oplus V(2)$ &for $n\geq 0$ & \begin{footnotesize}(See \cite{WILSON2})\end{footnotesize}& &\\
& \begin{footnotesize}{\it (See \cite{JMM}})\end{footnotesize}  & for $n\geq 4$& & & &\\
\hline




        \end{tabular}\hspace*{-1cm}

\end{table}
\end{center}
\end{landscape}
\subsection*{\bf Configuration space of $n$ points in $\mathbb{C}$  and the pure braid group.}  This example was the first studied with this approach in \cite[Example 5.1.A]{CEF}. We have  the co-$FI$-space $\mathcal{F}(\mathbb{C},\bullet)$ given by $\mathbf{n}\mapsto \mathcal{F}(\mathbb{C},n)$. Since these configuration spaces are Eilenberg-MacLane spaces, we could alternatively consider their fundamental groups $P_n$, the {\it pure braid groups}. 
The cohomology ring in this case is the so called Arnol'd algebra which is generated by cohomology classes of degree one. We recall its description in Table \ref{EXA} and recover the information about the weight and degree of generation from our discussion in Section \ref{FIREP}. Uniform representation stability was originally proved in \cite{CF}.

\subsection*{\bf Moduli space of $n$-pointed curves of genus zero.}
The {\it moduli space of $(n+1)$-pointed curves of genus zero} $\mathcal{M}_{0,n+1}$ has a natural action of $S_{n+1}$ by permuting the marked points. If we focus on the action of the subgroup $S_n$ generated by the transpositions $(2\ 3), (3\ 4),\ldots, (n\ n+1)$ in $S_{n+1}$, we can identify its cohomology ring with a subalgebra of $H^*(\mathcal{F}(\mathbb{C},n),\mathbb{Q})$ as we recall in Table \ref{EXA}. Hence the graded $FI$-algebra $H^*(\mathcal{M}_{0,\bullet+1},\mathbb{Q})$ is generated by degree one classes and our arguments in Section \ref{FIREP} readily apply. As explained in \cite{JIM3}, we can use conclusions on these ``shifted'' $FI$-modules to obtain information about $H^*(\mathcal{M}_{0,\bullet},\mathbb{Q})$.
\begin{remark}

The theory of $FI$-modules can also be used to study and obtain representation stability  of the cohomology of configuration spaces of $n$ points on connected manifolds  and moduli spaces of genus $g\geq 2$ with $n$ marked points (see \cite{CHURCH}, \cite{CEF}, \cite{JIM} and \cite{JIM2}). In these cases, little information is known about the corresponding cohomology rings and the arguments needed are more involved than the ones discussed in this note.  \end{remark}

\subsection*{\bf Pure virtual and flat braid groups.} 
The {\it pure virtual braid groups} $PvB_n$ and the {\it pure flat braid groups} $PfB_n$ are generalizations of the braid group $P_n$ that allow {\it virtual}  crossings of strands. Those virtual crossings were introduced in \cite{KAUFF}. The original motivation was to study knots in thickened
surfaces of higher genus  and the extension of knot theory to the domain of
Gauss codes and Gauss diagrams.  Explicit definitions and  presentations of these groups can be found in \cite[Section 2.1]{LEE2}.
As for the pure braid group $P_n$, there is a well defined action of $S_n$ on $H^*(PvB_n;\mathbb{Q})$ and $H^*(PfB_n;\mathbb{Q})$. Moreover, we have well-defined forgetful maps compatible with the $S_n$-action. Hence  we can consider the $FI$-modules $H^p(PvB_\bullet;\mathbb{Q})$ and $H^p(PfB_\bullet;\mathbb{Q})$ in the setting from Section \ref{FIREP}.

From the description in Table \ref{EXA}, we see that the corresponding cohomology rings have a graded $FI$-algebra structure and are generated by the first cohomology group. 
Moreover, from the explicit decompositions of $H^p(PvB_n,\mathbb{Q})$  and $H^p(PfB_n,\mathbb{Q})$ as a direct sum of induced representations obtained in \cite[Theorem 3 and 8]{LEE2}, it follows that $H^p(PvB_\bullet,\mathbb{Q})$ and $H^p(PfB_\bullet,\mathbb{Q})$ have the extra-structure of an $FI\#$-module (see \cite[Def. 4.1.1 and Thm. 4.1.5]{CEF}). This in particular implies that the corresponding stability degree is bounded above by the weight (\cite[Cor. 4.1.8]{CEF}). Theorem \ref{thm:FI} implies uniform representation stability, which was previously obtained in \cite{LEE2} using the approach in \cite{CF}.\bigskip

In \cite{WILSON1}, Wilson  extended the work of Church, Ellenberg, Farb, and Nagpal 
on sequences of representations of  $S_n$ to representations of the signed permutation groups $B_n$ and the even-signed permutation groups $D_n$. She introduced the notion of a finitely generated $FI_\mathcal{W}$-module, where $\mathcal{W}_n$ is the Weyl group in type $A_{n-1}$, $B_n/C_n$, or $D_n$. Arguments like the ones discussed in Section \ref{FIREP} also apply in this setting (\cite[Prop.  5.11]{WILSON1}). Uniform representation stability and polynomiality of characters are also consequences of this approach (\cite[Thms 4.26 \& 4.27]{WILSON1} and \cite[Thm 4.16]{WILSON2}).

\subsection*{\bf Group of pure string motions.}
The {\it group $\Sigma_n$ of  string motions} is another  generalization of the braid group. It is defined as the group of motions of $n$ smoothly embedded, oriented, unlinked, unknotted circles in $\mathbb{R}^3$ and can be  identified with the {\it symmetric
automorphism group} of the free group $F_n$. We refer the interested reader to \cite[Section 5.1]{WILSON2} and references therein.
The subgroup $P\Sigma_n$ of {\it pure string motions} or {\it pure symmetric automorphisms} is the analogue of the pure braid group and has a natural action of the signed permutation group $B_n$. The presentation of the cohomology ring that we recall in Table \ref{EXA} can be used to show that $H^*(P\Sigma_\bullet,\mathbb{Q})$ has the structure of what she calls a graded $FI_{BC}\#$-algebra (see \cite[Theorem 5.3]{WILSON2}) which turns out to be finitely generated by $H^1(P\Sigma_\bullet,\mathbb{Q})$. By restricting to the $S_n$-action, one recovers finite generation for $H^p(P\Sigma_\bullet,\mathbb{Q})$  as $FI$-modules.

\begin{remark} The cohomology groups of $P_n$, $PvB_n$, $PfB_n$ and $P\Sigma_n$ have the extra-structure of an $FI\#$-module. By \cite[4.1.7(g)]{CEF}, the corresponding characters are given by a unique polynomial character for all $n\geq 0$. In contrast, for  the spaces $M_n:=\overline{\mathcal{M}}_{0,n}(\mathbb{R})$ and $\mathcal{M}_{0,n+1}$, the formulas for the character of their first cohomology do not extend for the cases $n=0$ and $n=1$, respectively.  However, as the referee pointed out, it is interesting to notice that the formulas for $\chi_{H^1(M_n)}$ and $\chi_{H^1(\mathcal{M}_{0,n+1})}$ are accurate once the spaces considered are non-empty ($n\geq 3$ and $n\geq 2$, respectively). For instance, the formula for $\chi_{H^1(M_n)}$ is identically $0$ as a class function on $S_n$ for $n = 3$ which corresponds to the fact that $H^1(M_3;\mathbb{Q})=0$. And the formula for $\chi_{H^1(\mathcal{M}_{0,n+1})}$  works for all $n\geq 2$ (see for example \cite[Formula (2)]{JIM3}).

\end{remark}

\subsection*{\bf Hyperplane complements.} Another generalization of  configuration spaces on the plane is given by hyperplane complements for other Weyl groups. Consider the canonical action of $\mathcal{W}_n$ on $\mathbb{C}^n$  by (signed) permutation matrices and let $\mathcal{A}(n)$ be the set of hyperplanes fixed by the (complexified) reflections of $\mathcal{W}_n$.  The sequence of hyperplane complements  $\mathcal{M}_\mathcal{W}(n):=\mathbb{C}^n-\bigcup_{H\in\mathcal{A}(n)}H$ can be encoded as a co-$FI_\mathcal{W}$-space $\mathcal{M}_\mathcal{W}(\bullet)$. In particular $\mathcal{M}_{A}(\bullet)=\mathcal{F}(\mathbb{C},\bullet)$. The corresponding fundamental groups are referred as ``generalized pure braid groups''. The work  in \cite{BRIESKORN} and \cite{ORSOL} show that the cohomology ring for $\mathcal{M}_\mathcal{W}(n)$ is also generated by cohomology classes of degree one and arguments as before apply. The details are worked out in \cite[Thm 5.8]{WILSON2}. 
\subsection*{\bf Complex points of the Deligne-Mumford compatification of $\mathcal{M}_{0,n}$ } 
Another related example is  the complex locus  of the variety $\overline{\mathcal{M}}_{0,n}$, that we denote by $\overline{\mathcal{M}}_{0,n}(\mathbb{C})$. From our discussion in Section \ref{REALMODULI}, it is clear that $H^*(\overline{\mathcal{M}}_{0,\bullet}(\mathbb{C});\mathbb{Q})$ has the struture of a graded $FI$-algebra. Moreover, from the results in \cite{KEEL} it follows that it is generated by the $FI$-module  $H^2(\overline{\mathcal{M}}_{0,\bullet}(\mathbb{C});\mathbb{Q})$. However, the computations in \cite{KEEL} also imply that the dimension of $H^i(\overline{\mathcal{M}}_{0,n}(\mathbb{C});\mathbb{Q})$ is  exponential in $n$. Therefore, from ~\autoref{thm:CHARCACTUS} it follows  that the $FI$-modules $H^i(\overline{\mathcal{M}}_{0,\bullet}(\mathbb{C});\mathbb{Q})$ are not finitely generated.






\bibliographystyle{amsalpha}
\bibliography{referFI}



 

 





\end{document}